\documentclass[10pt,a4paper,twoside]{amsart}

\usepackage{graphicx}
\usepackage{a4wide}

\newtheorem{definition}{Definition}[section]
\newtheorem{theorem}{Theorem}[section]
\newtheorem{lemma}{Lemma}[section]

\newtheorem*{maintheorem*}{Main Theorem}
\allowdisplaybreaks
\numberwithin{equation}{section}

\newcommand{\norm}[1]{\left\| #1 \right\|}

\newcommand{\eps}{\varepsilon}

\newcommand{\eb}{{\eps,\beta}}
\newcommand{\ueb}{u_\eb}

\newcommand{\Peb}{P_{\eps,\beta}}
\newcommand{\Feb}{F_{\eps,\beta}}
\newcommand{\Pe}{P_\eps}

\newcommand{\Fe}{F_\eps}

\newcommand{\pt}{\partial_t}

\newcommand{\px}{\partial_x }
\newcommand{\pxx}{\partial_{xx}^2}
\newcommand{\pxxx}{\partial_{xxx}^3}

\newcommand{\ptx}{\partial_{tx}^2}

\renewcommand{\i}{\ifmmode\mathit{\mathchar"7010 }\else\char"10 \fi}
\renewcommand{\j}{\ifmmode\mathit{\mathchar"7011 }\else\char"11 \fi}
\newcommand{\R}{\mathbb{R}}
\newcommand{\N}{\mathbb{N}}

\newcommand{\supp}{\mathrm{supp}\,}

{%

\begin{enumerate}}%
{\end{enumerate}}

{%

\begin{enumerate}}%
{\end{enumerate}}

\begin{document}\large

\title[The regularized short pulse equation ]{Convergence of the regularized short pulse equation \\ to the short pulse One}

\author[G. M. Coclite and L. di Ruvo]{Giuseppe Maria Coclite and Lorenzo di Ruvo}
\address[Giuseppe Maria Coclite and Lorenzo di Ruvo]
{\newline Department of Mathematics,   University of Bari, via E. Orabona 4, 70125 Bari,   Italy}
\email[]{giuseppemaria.coclite@uniba.it, lorenzo.diruvo@uniba.it}
\urladdr{http://www.dm.uniba.it/Members/coclitegm/}

\date{\today}



%35G25 Initial value problems for nonlinear higher-order PDE, nonlinear evolution equations
%35L05 Wave equation
%74S20 Finite difference methods
%35L65 Conservation laws
%65M12 Stability and convergence of numerical methods

\thanks{The authors are members of the Gruppo Nazionale per l'Analisi Matematica, la Probabilit\`a e le loro Applicazioni (GNAMPA) of the Istituto Nazionale di Alta Matematica (INdAM)}

\begin{abstract}
We consider the regularized short-pulse equation, which contains nonlinear dispersive effects. We prove
that as the diffusion parameter tends to zero, the solutions of the dispersive equation converge to discontinuous weak solutions of the short-pulse one.
The proof relies on deriving suitable a priori estimates together with an application of the compensated compactness method in the $L^p$ setting.
\end{abstract}

\maketitle

%\tableofcontents

\section{Introduction}
\label{sec:intro}

The nonlinear evolution equation
\begin{equation}
\label{eq:RSHP}
\px\left(\pt u-\frac{1}{6}\px (u^3)-\beta \pxxx u\right)=\gamma u, \quad \gamma>0,
\end{equation}
known as the regularized short pulse equation, was derived by Costanzino, Manukian and Jones \cite{CMJ}
in the context of the nonlinear Maxwell equations with high-frequency dispersion. Mathematical properties of the regularized short pulse equation \eqref{eq:RSHP} were studied
recently in many details, including the local and global well-posedness in energy space
\cite{CMJ,PS}, and stability of solitary waves \cite{CMJ}.
We rewrite \eqref{eq:RSHP} in the following way
\begin{equation}
\label{eq:RSHP-u}
\begin{cases}
\displaystyle \pt u -\frac{1}{6}\px (u^3)- \beta \pxxx u=\gamma \int^{x}_{0} u(t,y) dy,&\qquad t>0, \ x\in\R,\\
u(0,x)=u_0(x), &\qquad x\in\R,
\end{cases}
\end{equation}
or equivalently,
\begin{equation}
\label{eq:RSHP-p}
\begin{cases}
\displaystyle \pt u-\frac{1}{6} \px (u^3) - \beta\pxxx u=\gamma P,&\qquad t>0, \ x\in\R ,\\
\px P=u,&\qquad t>0, \ x\in\R,\\
 P(t,0)=0,& \qquad t>0,\\
 u(0,x)=u_0(x), &\qquad x\in\R.
\end{cases}
\end{equation}
We are interested in the no high frequency limit,  i.e., we send $\beta\to 0$ in \eqref{eq:RSHP}. In this way, we pass from \eqref{eq:RSHP} to the equation
\begin{equation}
\label{eq:SHP}
\begin{cases}
\displaystyle \px\left(\pt u-\frac{1}{6}\px (u^3)\right)=\gamma u,&\qquad t>0, \ x\in\R ,\\
u(0,x)=u_0(x), &\qquad x\in\R.
\end{cases}
\end{equation}
The equation \eqref{eq:SHP} is known as the short pulse equation, and was introduced recently by
Sch\"afer and Wayne \cite{SW} as  a model equation describing the propagation of ultra-short light pulses in silica optical fibers.
It provides also an approximation of nonlinear wave packets in dispersive media in the limit of few cycles on the ultra-short pulse scale.
Numerical simulations \cite{CJSW} show that the short pulse equation approximation 
to Maxwell's equations in the case when the pulse spectrum is not narrowly localized 
around the carrier frequency is better than the one obtained from the nonlinear Schr\"odinger 
equation, which models the evolution of slowly varying wave trains.
Such ultra-short plays a key role in the development of future technologies of ultra-fast optical transmission of informations.

In \cite{B} the author studied a new hierarchy of equations containing the short pulse equation \eqref{eq:SHP} and the elastic 
beam equation, which describes nonlinear transverse oscillations of elastic beams 
under tension. He showed that the hierarchy of equations is integrable. He obtained the two compatible Hamiltonian structures
and constructed an infinite series of both local and nonlocal conserved charges. Moreover, he gave the Lax description for both systems. 
The integrability and the existence of solitary wave solutions have been studied in \cite{SS, SS1}.
Well-posedness and wave breaking for the short pulse equation have been 
studied in \cite{SW} and \cite{LPS}, respectively. 

We remark here that \eqref{eq:RSHP} and \eqref{eq:SHP} look similar to the Ostrovsky and Ovstrosky-Hunter equations
\begin{equation}
\label{eq:OH}
\px \left(\pt u+\frac{1}{2}\px (u^2)-\beta\pxxx u\right)=\gamma u,\qquad
\px\left(\pt u+\frac{1}{2}\px (u^2)\right)=\gamma u,
\end{equation}
which was derived by Ostrovsky \cite{O} and Hunter and Tan \cite{HT} as a model for internal solitary waves in the ocean with rotation effects of strength $\gamma$. 
The deep difference is in the power and the sing of the non linear term.
The well-posedness on the context of discontinuous solutions of the Ovstrosky-Hunter equation has been studied in \cite{Cd, CdRK, CdK, dR}.
In \cite{Cd2} we proved the convergence of the solutions of the Ostrovsky equation to the discontinuous ones of the Ostrovsky-Hunter equation.

Integrating \eqref{eq:SHP} with respect to $x$ we gain the integro-differential formulation of  \eqref{eq:SHP} (see \cite{SS})
\begin{equation*}
\pt u -\px \left(\frac{u^3}{6}\right)=\gamma \int^x_0 u(t,y) dy,
\end{equation*}
that is equivalent to
\begin{equation}
\label{eq:OHwP}
\pt u-\px \left(\frac{u^3}{6}\right)=\gamma P,\qquad \px P=u,\quad P(\cdot,0)=0,\quad u(0,\cdot)=u_0.
\end{equation}
On the initial datum, we assume that
\begin{equation}
\label{eq:assinit}
u_0\in L^2(\R)\cap L^{6}(\R),\quad\int_{\R}u_{0}(x)dx=0,
\end{equation}
and on the function
\begin{equation}
\label{eq:def-di-P0}
P_{0}(x)=\int_{-\infty}^{x} u_{0}(y)dy, \quad x\in\R,
\end{equation}
we assume that
\begin{equation}
\label{eq:L-2P0}
\begin{split}
\norm{P_0}^2_{L^2(\R)}&=\int_{\R}\left(\int_{-\infty}^{x}u_{0}(y)dy\right)^2dx <\infty,\\
\int_{\R}P_0(x)dx&= \int_{\R}\left(\int_{-\infty}^{x}u_{0}(y)dy\right)dx=0.
\end{split}
\end{equation}
\begin{definition}
\label{def:sol}
We say that $u\in  L^{\infty}((0,T)\times\R),\,T>0,$ is an entropy solution of the initial
value problem \eqref{eq:SHP} if
\begin{itemize}
\item[$i$)] $u$ is a distributional solution of \eqref{eq:OHwP};
\item[$ii$)] for every convex function $\eta\in  C^2(\R)$ the
entropy inequality
\begin{equation}
\label{eq:OHentropy}
\pt \eta(u)+ \px q(u)-\gamma\eta'(u) P\le 0, \qquad     q(u)=-\int^u \frac{\xi^2}{2} \eta'(\xi)\, d\xi
\end{equation}
holds in the sense of distributions in $(0,\infty)\times\R$.
\end{itemize}
\end{definition}
In \cite{Cd1}, it is proved that \eqref{eq:SHP} has an unique entropy solution in the sense of Definition \ref{def:sol} that depends Lipschitz continuously on the initial condition.

We study the dispersion-diffusion limit for \eqref{eq:RSHP}. Therefore, we fix two small numbers $0<\eps,\, \beta <1$ and consider the following third order problem
\begin{equation}
\label{eq:OHepsw}
\begin{cases}
\displaystyle \pt \ueb -\frac{1}{6}\px \ueb^3 -\beta\pxxx\ueb=\gamma\Peb+ \eps\pxx\ueb,&\quad t>0,\ x\in\R,\\
-\eps\pxx\Peb+\px\Peb=\ueb,&\quad t>0,\ x\in\R,\\
\Peb(t,0)=0,&\quad t>0,\\
\ueb(0,x)=u_{\eps,\beta,0}(x),&\quad x\in\R,
\end{cases}
\end{equation}
where $u_{\eps,\beta,0}$ is a $C^\infty$ approximation of $u_{0}$ such that
\begin{equation}
\begin{split}
\label{eq:u0eps}
&\norm{u_{\eps,\beta, 0}}^2_{L^2(\R)}+ \norm{u_{\eps,\beta, 0}}^6_{L^6(\R)}+(\beta+ \eps^2) \norm{\px u_{\eps,\beta,0}}^2_{L^2(\R)}\le C_{0}, \quad \eps,\beta >0,\\
& \int_{\R} u_{\eps,\beta,0}(x) dx =0, \quad \eps,\beta >0,\\
& \int_{\R}P_{0,\eps, \beta}(x)dx=0, \quad \eps,\beta>0,\\
& \norm{P_{\eps,\beta,0}}^2_{L^2(\R)}+\eps^2 \norm{\px P_{\eps,\beta,0}}^2_{L^2(\R)}\le C_{0}, \quad \eps,\beta >0,
\end{split}
\end{equation}
where $C_0$ is a constant independent on $\eps$ and $\beta$.

The main result of this paper is the following theorem.

\begin{theorem}
\label{th:main}
Assume that \eqref{eq:assinit}, \eqref{eq:def-di-P0}, \eqref{eq:L-2P0}, and \eqref{eq:u0eps} hold.
If
\begin{equation}
\label{eq:beta-eps}
\beta=\mathbf{\mathcal{O}}(\eps^2),
\end{equation}
then, there exist two sequences $\{\eps_{n}\}_{n\in\N}$, $\{\beta_{n}\}_{n\in\N}$, with $\eps_n, \beta_n \to 0$, and two limit functions
\begin{align*}
&u\in L^{\infty}(0,T; L^2(\R)\cap L^6(\R)), \quad T>0,\\
&P\in L^{\infty}((0,T)\times\R)\cap L^{2}((0,T)\times\R), \quad T>0,
\end{align*}
such that
\begin{itemize}
\item[$i)$] $u$ is  a distributional solution of \eqref{eq:SHP},
\item[$ii)$] $u_{\eps_n, \beta_n}\to u$ strongly in $L^{p}_{loc}((0,\infty)\times\R)$, for each $1\le p <6$,
\item[$iii)$] $P_{\eps_n, \beta_n}\to P$ strongly in $L^{\infty}((0,T)\times\R)\cap L^{2}((0,T)\times\R)$, $T>0$.
\end{itemize}
Moreover, if
\begin{equation}
\label{eq:beta-eps-1}
\beta=o(\eps^2),
\end{equation}
then,
\begin{itemize}
\item[$iv)$] $u$ is  the unique entropy solution of \eqref{eq:SHP},
\item[$v)$] $P_{\eps_n, \beta_n}\to P$ strongly in $L^{p}_{loc}(0,\infty;W^{1,\infty}(\R)\cap H^{1}(\R))$, for each $1\le p < 6$.
\end{itemize}
In particular, we have
\begin{equation}
\label{eq:umedianulla}
\int_{\R} u(t,x)dx =0, \quad t>0.
\end{equation}
\end{theorem}

The paper is organized in two sections. In Section \ref{sec:vv}, we prove some a priori estimates, while in Section \ref{sec:theor} we prove Theorem \ref{th:main}.

\section{A priori Estimates}
\label{sec:vv}

This section is devoted to some a priori estimates on $\ueb$. We denote with $C_0$ the constants which depend only on the initial data, and with $C(T)$, the constants which depend also on $T$.

Arguing as \cite[Lemmas $3.1$, $3.2$, $3.4$, $3.5$]{Cd1}, we obtain the following result
\begin{lemma}\label{lm:15}
\label{lm:cns-1}
For each $t\in (0,\infty)$,
\begin{align}
\label{eq:P-pxP-intfy-1}
\Peb(t,-\infty)=\px\Peb(t,-\infty)=\Peb(t,\infty)=\px\Peb(t,\infty)=&0,\\
\label{eq:equ-L2-stima-1}
\eps^2\norm{\pxx\Peb(t,\cdot)}^2_{L^2(\R)}+ \norm{\px\Peb(t,\cdot)}^2_{L^2(0,\infty)}=&\norm{\ueb(t,\cdot)}^2_{L^2(0,\infty)},\\
\label{eq:int-u-1}
\int_{\R}\ueb(t,x) dx =&0,\\
\label{eq:L-infty-Px-1}
\sqrt{\eps}\norm{\px\Peb(t, \cdot)}_{L^{\infty}(\R)}\le& \norm{\ueb(t,\cdot)}_{L^2(\R)},\\
\label{eq:uP-1}
\int_{\R}\ueb(t,x)\Peb(t,x) dx\le& \norm{\ueb(t,\cdot)}^2_{L^2(\R)}.
\end{align}
Moreover, the inequality holds
\begin{equation}
\label{eq:stima-l-2-1}
\norm{\ueb(t,\cdot)}^2_{L^2(\R)}+ 2\eps e^{2\gamma t}\int_{0}^{t} e^{-2\gamma s}\norm{\px\ueb(s,\cdot)}^2_{L^2(\R)}ds\le e^{2\gamma t}C_{0}.
\end{equation}
In particular, we have
\begin{equation}
\label{eq:10021-1}
\begin{split}
\eps\norm{\pxx\Peb(t,\cdot)}_{L^2(\R)}, \norm{\px\Peb(t,\cdot)}_{L^2(\R)}&\le e^{\gamma t}C_{0},\\
\sqrt{\eps}\norm{\px\Peb(t, \cdot)}_{L^{\infty}(\R)}&\le e^{\gamma t}C_{0},\\
\sqrt{\eps}\vert\px\Peb(t,0)\vert &\le e^{\gamma t}C_{0}.
\end{split}
\end{equation}
\end{lemma}
Arguing as \cite[Lemma $2.3$]{Cd2}, we have can prove the following lemma.
\begin{lemma}
\label{lm:p8}
For each $t\ge 0$, we have that
\begin{align}
\label{eq:intp-infty}
\int_{0}^{-\infty}\Peb(t,x)dx&=a_{\eps,\beta}(t), \\
\label{eq:int+infty}
\int_{0}^{\infty}\Peb(t,x)dx&=a_{\eps,\beta}(t),
\end{align}
where
\begin{equation*}
a_{\eps,\beta}(t)= \frac{1}{\gamma}\left(\eps\ptx\Peb(t,0)+\frac{1}{6}\ueb^3(t,0)+\beta\pxx\ueb(t,0)+\eps\px\ueb(t,0)\right).
\end{equation*}
Moreover,
\begin{equation}
\label{eq:Pmedianulla}
\int_{\R}\Peb(t,x)dx=0, \quad t\geq 0.
\end{equation}
\end{lemma}
Lemma \ref{lm:p8} says that $\Peb(t,\cdot)$  is integrable at $\pm\infty$.
Therefore, for each $t\ge 0$, we can consider the following function
\begin{equation}
\label{eq:F1}
\Feb(t,x)=\int_{-\infty}^{x}\Pe(t,y)dy.
\end{equation}
Thus, integrating the first equation in \eqref{eq:OHepsw} on $(-\infty,x)$, thanks to \eqref{eq:F1}, we  obtain that
\begin{equation}
\label{eq:intP1}
\int_{-\infty}^{x}\pt\ueb dy-\frac{1}{6}\ueb^3 - \beta\pxx\ueb-\eps\px\ueb=\gamma\Feb.
\end{equation}
Thanks to \eqref{eq:P-pxP-intfy-1}, integrating on $\R$  the second equation in \eqref{eq:OHepsw}, we have
\begin{equation}
\label{eq:equaP}
\Peb(t,x)= \int_{-\infty}^{x} \ueb(t,y)dy +\eps\px\Pe(t,x).
\end{equation}
Differentiating \eqref{eq:equaP} with respect to $t$, we get
\begin{equation}
\label{eq:equap-1}
\pt\Peb(t,x)= \int_{-\infty}^{x} \pt\ueb(t,y)dy +\eps\ptx\Peb(t,x).
\end{equation}
Therefore, \eqref{eq:intP1} and \eqref{eq:equap-1} give
\begin{equation}
\label{eq:equat-P}
\pt\Peb-\eps\ptx\Peb -\frac{1}{6}\ueb^3 - \beta\pxx\ueb-\eps\px\ueb=\gamma\Feb.
\end{equation}

\begin{lemma}\label{lm:u-infty}
Fixed $T>0$. There exists a function $C(T)>0$, independent on $\eps$ and $\beta$, such that
\begin{equation}
\label{eq:l-infty-u}
\norm{\ueb}_{L^{\infty}((0,T)\times\R)}\le C(T)\beta^{-\frac{1}{2}}.
\end{equation}
Moreover, for every $0\le t\le T$,
\begin{equation}
\label{eq:ux2}
\begin{split}
\beta\norm{\px\ueb(t,\cdot)}^2_{L^2(\R)} &+\gamma\norm{\Peb(t,\cdot)}^2_{L^2(\R)}+ \eps^2\gamma \norm{\px\Peb(t,\cdot)}^2_{L^2(\R)} \\
&+ \beta\eps e^{2\gamma t}\int_{0}^{t}e^{-2 \gamma s}\norm{\pxx\ueb(s,\cdot)}^2_{L^2(\R)}ds\le C(T)\beta^{-2}.
\end{split}
\end{equation}
\end{lemma}
\begin{proof}
Let $0\le t\le T$. Multiplying  \eqref{eq:OHepsw} by $\displaystyle -\beta\pxx\ueb -\frac{1}{6} \ueb^3$, we have
\begin{equation}
\label{eq:Ohmp}
\begin{split}
\left(-\beta\pxx\ueb -\frac{1}{6}\ueb^3\right)\pt\ueb &- \frac{1}{6}\left(-\beta\pxx\ueb -\frac{1}{6}\ueb^3\right)\px\ueb^3\\
&-\beta\left(-\beta\pxx\ueb -\frac{1}{6}\ueb^3\right)\pxxx\ueb\\
=&\gamma\left(-\beta\pxx\ueb -\frac{1}{6} \ueb^3\right)\Peb\\
&+\eps\left(-\beta\pxx\ueb -\frac{1}{6}\ueb^3\right)\pxx\ueb.
\end{split}
\end{equation}
Observe that
\begin{align*}
\int_{\R}\left(-\beta\pxx\ueb -\frac{1}{6}\ueb^3\right)\pt\ueb dx &= \frac{d}{dt}\left(\frac{\beta}{2}\norm{\px\ueb(t,\cdot)}^2_{L^2(\R)}-\frac{1}{24}\int_{\R}\ueb^4dx\right),\\
- \frac{1}{6}\int_{\R}\left(-\beta\pxx\ueb -\frac{1}{6}\ueb^3\right)\px \ueb^3 dx &=\frac{\beta}{2}\int_{\R}\ueb^2\px\ueb\pxx\ueb dx,\\
-\beta\int_{\R}\left(-\beta\pxx\ueb -\frac{1}{6}\ueb^3\right)\pxxx\ueb dx &= -\frac{\beta}{2}\int_{\R}\ueb^2\px\ueb\pxx\ueb dx,\\
\eps\int_{\R} \left(-\beta\pxx\ueb -\frac{1}{6}\ueb^3\right)\pxx\ueb dx &=  -\eps\beta \norm{\pxx\ueb(t,\cdot)}^2_{L^2(\R)} + \frac{\eps}{2}\int_{\R}\ueb^2(\px\ueb)^2dx.
\end{align*}
Moreover, for \eqref{eq:P-pxP-intfy-1},
\begin{equation*}
-\gamma\beta\int_{\R}\pxx\ueb\Peb dx = \gamma\beta\int_{\R}\px\ueb\px\Peb dx.
\end{equation*}
Therefore, integrating \eqref{eq:Ohmp} over $\R$,
\begin{equation}
\label{eq:p12}
\begin{split}
&\frac{d}{dt}\left(\beta\norm{\px\ueb(t,\cdot)}^2_{L^2(\R)}-\frac{1}{12}\int_{\R}\ueb^4dx\right) + 2\beta\eps\norm{\pxx\ueb(t,\cdot)}^2_{L^2(\R)}\\
&\quad= 2\gamma\beta\int_{\R}\px\ueb\px\Peb dx -\frac{\gamma}{3}\int_{\R} \ueb^3 \Peb dx+\eps\int_{\R}\ueb^2(\px\ueb)^2dx.
\end{split}
\end{equation}
Multiplying \eqref{eq:equat-P} by $2\gamma\Peb - 2\eps\gamma\px\Peb$, we have
\begin{equation}
\label{eq:P12}
\begin{split}
&(\pt\Peb-\eps\ptx\Peb)(2\gamma\Peb - 2\eps\gamma\px\Peb)\\
&\quad = \gamma\Feb(2\gamma\Peb - 2\eps\gamma\px\Peb)+\frac{1}{6}\ueb^3(2\gamma\Peb - 2\eps\gamma\px\Peb)\\
&\qquad+\beta\pxx\ueb(2\gamma\Peb - 2\eps\gamma\px\Peb)+\eps\px\ueb(2\gamma\Peb - 2\eps\gamma\px\Peb).
\end{split}
\end{equation}
It follows from \eqref{eq:P-pxP-intfy-1} that
\begin{equation}
\begin{split}
\label{eq:P14}
-2\eps\gamma\int_{\R}\Peb\ptx\Peb dx =& 2 \eps \gamma \int_{\R}\px\Peb\pt\Peb dx,\\
2\beta\gamma\int_{\R}\pxx\ueb\Peb dx =& -2\beta\gamma\int_{\R}\px\ueb \px\Peb dx.
\end{split}
\end{equation}
Therefore, arguing as \cite[Lemma $2.6$]{Cd1}, an integration on $\R$ of \eqref{eq:P12}, \eqref{eq:P-pxP-intfy-1} and \eqref{eq:P14} give
\begin{equation}
\label{eq:P15}
\begin{split}
&\frac{d}{dt}\left(\gamma\norm{\Peb(t,\cdot)}^2_{L^2(\R)} + \eps^2\gamma \norm{\px\Peb(t,\cdot)}^2_{L^2(\R)}\right)\\
&\quad =2\gamma^2\int_{\R}\Feb\Peb dx -2\eps\gamma^2\int_{\R} \Feb\px\Peb dx\\
&\qquad +\frac{\gamma}{3}\int_{\R}\ueb^3\Peb dx -\frac{\eps\gamma}{3}\int_{\R}\ueb^3\px\Peb dx\\
&\qquad -2\beta\gamma\int_{\R}\px\ueb\px\Peb dx -2\beta\eps\int_{\R}\pxx\ueb\Peb dx\\
&\qquad + 2\eps\gamma\int_{\R}\px\ueb\Peb dx -2\eps^2\gamma\int_{\R}\px\ueb\px\Peb dx.
\end{split}
\end{equation}
Due to \eqref{eq:Pmedianulla} and \eqref{eq:F1},
\begin{equation}
\label{eq:PF}
\begin{split}
2\gamma^2\int_{\R}\Feb\Peb dx=&2\gamma^2\int_{\R}\Feb\px\Feb dx \\
= &\gamma^2(\Feb(t,\infty))^2=\gamma^2\left(\int_{\R}\Peb(t,x) dx \right)^2=0.
\end{split}
\end{equation}
It follows from \eqref{eq:P15} and \eqref{eq:PF} that
\begin{equation}
\label{eq:P16}
\begin{split}
&\frac{d}{dt}\left(\gamma\norm{\Peb(t,\cdot)}^2_{L^2(\R)} + \eps^2\gamma \norm{\px\Peb(t,\cdot)}^2_{L^2(\R)}\right)\\
&\quad = -2\eps\gamma^2\int_{\R} \Feb\px\Peb dx+\frac{\gamma}{3}\int_{\R}\ueb^3\Peb dx\\
&\qquad  -\frac{\eps\gamma}{3}\int_{\R}\ueb^3\px\Peb dx-2\beta\gamma\int_{\R}\px\ueb\px\Peb dx\\
&\qquad  -2\beta\eps\int_{\R}\px\ueb\px\Peb dx + 2\eps\gamma\int_{\R}\px\ueb\Peb dx\\
&\qquad  -2\eps^2\gamma\int_{\R}\px\ueb\px\Peb dx.
\end{split}
\end{equation}
Adding \eqref{eq:p12} and \eqref{eq:P16}, we get
\begin{equation}
\label{eq:p234}
\begin{split}
&\frac{d}{dt}G_{1}(t) + 2\beta\eps\norm{\pxx\ueb(t,\cdot)}^2_{L^2(\R)}\\
&\quad=  \eps\int_{\R}\ueb^2(\px\ueb)^2dx-2\eps\gamma^2\int_{\R} \Feb\px\Peb dx\\
&\qquad-\frac{\eps\gamma}{3}\int_{\R}\ueb^3\px\Peb dx-2\eps\beta\int_{\R}\pxx\ueb\Peb dx  \\
&\qquad +2\eps\gamma\int_{\R}\px\ueb\Peb dx -2\eps^2\gamma\int_{\R}\px\ueb\px\Peb dx ,
\end{split}
\end{equation}
where
\begin{equation}
\label{eq:def-di-G}
\begin{split}
G_{1}(t)=\beta\norm{\px\ueb(t,\cdot)}^2_{L^2(\R)}&-\frac{1}{12}\int_{\R}\ueb^4dx\\
&+\gamma\norm{\Peb(t,\cdot)}^2_{L^2(\R)} + \eps^2\gamma \norm{\px\Peb(t,\cdot)}^2_{L^2(\R)}.
\end{split}
\end{equation}
Since $0<\eps <1$, thanks to \eqref{eq:P-pxP-intfy-1} and \eqref{eq:F1},
\begin{equation}
\label{eq:F-pxP}
-2\eps\gamma^2\int_{\R} \Feb\px\Peb dx=2\eps\gamma^2 \int_{\R}\px\Feb\Peb dx \le 2\gamma^2\norm{\Peb(t,\cdot)}^2_{L^2(\R)}.
\end{equation}
Since $0<\eps,\beta <1$, due to \eqref{eq:P-pxP-intfy-1}, \eqref{eq:10021-1} and the Young inequality,
\begin{equation}
\label{eq:you1}
\begin{split}
&-2\eps\beta\int_{\R}\pxx\ueb\Peb dx=2\eps\beta\int_{\R}\px\ueb\px\Peb dx\\
&\quad\le 2\eps\beta\left\vert\int_{\R} \px\ueb\px\Peb dx\right\vert\\
&\quad\le 2\eps\int_{\R}\vert\px\ueb\vert\vert\px\Pe\vert dx\\
&\quad\le\eps\norm{\px\ueb(t,\cdot)}^2_{L^2(\R)} +\eps\norm{\Peb(t,\cdot)}^2_{L^2(\R)}\\
&\quad\le \eps\norm{\px\ueb(t,\cdot)}^2_{L^2(\R)} +  \norm{\Peb(t,\cdot)}^2_{L^2(\R)},\\
&\quad\le \eps\norm{\px\ueb(t,\cdot)}^2_{L^2(\R)} + C_{0}e^{2\gamma t}\\
&\quad\le  \eps \norm{\px\ueb(t,\cdot)}^2_{L^2(\R)} + C(T),\\
&2\eps\gamma\int_{\R}\px\ueb\Peb dx  =- 2\eps\gamma\int_{\R}\ueb\px\Peb dx\\
&\quad\le2\eps\gamma\left\vert\int_{\R}\ueb\px\Peb dx \right\vert\\
&\quad\le 2\eps\gamma\int_{\R}\vert\ueb\vert\vert\px\Peb\vert dx\\
&\quad\le  \eps\gamma\norm{\ueb(t,\cdot)}^2_{L^2(\R)} + \eps\gamma\norm{\px\Peb(t,\cdot)}^2_{L^2(\R)}\\
&\quad\le  \gamma\norm{\ueb(t,\cdot)}^2_{L^2(\R)} + \gamma\norm{\px\Peb(t,\cdot)}^2_{L^2(\R)}\\
&\quad\le  \gamma C_{0}e^{2\gamma t} + \gamma C_{0}e^{2\gamma t} \le C(T).
\end{split}
\end{equation}
For the Young inequality,
\begin{equation}
\label{eq:you2}
\begin{split}
2\eps^2\gamma\left\vert\int_{\R}\px\ueb\px\Peb dx\right\vert\le&\eps^2\int_{\R}\vert\px\ueb\vert\vert 2\gamma\px\Peb\vert dx\\
\le& \frac{\eps^2}{2}\norm{\px\ueb(t,\cdot)}^2_{L^2(\R)}+ 2\eps^2\gamma^2\norm{\px\Peb(t,\cdot)}^2_{L^2(\R)}\\
\le& \frac{\eps}{2}\norm{\px\ueb(t,\cdot)}^2_{L^2(\R)}+ 2\eps^2\gamma^2\norm{\px\Peb(t,\cdot)}^2_{L^2(\R)}.
\end{split}
\end{equation}
It follows from \eqref{eq:p234}, \eqref{eq:def-di-G}, \eqref{eq:F-pxP}, \eqref{eq:you1} and \eqref{eq:you2} that
\begin{align*}
&\frac{d}{dt}G_{1}(t) +2\beta\eps\norm{\pxx\ueb(t,\cdot)}^2_{L^2(\R)}\\
&\quad \le 2\gamma^2\norm{\Peb(t,\cdot)}^2_{L^2(\R)}+ 2\gamma^2\eps^2\norm{\px\Peb(t,\cdot)}^2_{L^2(\R)}\\
&\qquad +\eps\int_{\R}\ueb^2(\px\ueb)^2dx +\frac{\eps\gamma}{3}\int_{\R}\vert\ueb\vert^3\vert\vert\px\Peb\vert dx\\
&\qquad +\frac{3\eps}{2}\norm{\px\ueb(t,\cdot)}^2_{L^2(\R)}+C(T)\\
&\quad\le 2\gamma\beta\norm{\px\ueb(t,\cdot)}^2_{L^2(\R)} + 2\gamma^2\norm{\Peb(t,\cdot)}^2_{L^2(\R)}\\
&\qquad + 2\gamma^2\eps^2\norm{\px\Peb(t,\cdot)}^2_{L^2(\R)}+\eps\int_{\R}\ueb^2(\px\ueb)^2dx \\
&\qquad +\frac{\eps\gamma}{3}\int_{\R}\vert\ueb\vert^3\vert\vert\px\Peb\vert dx - \frac{\gamma}{6}\int_{\R}\ueb^4dx\\
&\qquad + \frac{\gamma}{6}\int_{\R}\ueb^4dx + \frac{3\eps}{2} \norm{\px\ueb(t,\cdot)}^2_{L^2(\R)}+C(T)\\
&\quad = 2\gamma G_{1}(t) +\eps\int_{\R}\ueb^2(\px\ueb)^2dx +\frac{\eps\gamma}{3}\int_{\R}\vert\ueb\vert^3\vert\vert\px\Peb\vert dx\\
&\qquad+ \frac{\gamma}{6}\int_{\R}\ueb^4dx + \frac{3\eps}{2} \norm{\px\ueb(t,\cdot)}^2_{L^2(\R)}+C(T),
\end{align*}
that is
\begin{equation}
\label{eq:P123}
\begin{split}
\frac{d}{dt}G_{1}(t) &- 2\gamma  G_{1}(t) +2\beta\eps\norm{\pxx\ueb(t,\cdot)}^2_{L^2(\R)}\\
\le & \eps\int_{\R}\ueb^2(\px\ueb)^2dx +\frac{\eps\gamma}{3}\int_{\R}\vert\ueb\vert^3\vert\vert\px\Peb\vert dx\\
&+\frac{\gamma}{6}\int_{\R}\ueb^4dx + \frac{3\eps}{2} \norm{\px\ueb(t,\cdot)}^2_{L^2(\R)} +C(T).
\end{split}
\end{equation}
Since $0<\eps<1$, due to \eqref{eq:stima-l-2-1}, \eqref{eq:10021-1} and the Young inequality,
\begin{equation}
\label{eq:you4}
\begin{split}
&\frac{\eps\gamma}{3}\int_{\R}\vert\ueb\vert^3\vert\vert\px\Peb\vert dx\le \frac{\eps\gamma}{3} \int_{\R}\ueb^2(\px\Peb)^2dx + \frac{\eps\gamma}{3}\int_{\R} \ueb^4 dx\\
&\quad \le \frac{\gamma}{3}\eps \norm{\px\Peb(t,\cdot)}^2_{L^{\infty}(\R)}\norm{\ueb(t,\cdot)}^2_{L^2(\R)} + \frac{\gamma}{3}\norm{\ueb}^2_{L^{\infty}((0,T)\times\R)}\norm{\ueb(t,\cdot)}^2_{L^2(\R)}\\
&\quad \le \frac{\gamma}{3}C_{0}e^{4\gamma t} + \frac{\gamma}{3} C_{0}e^{2\gamma t}\norm{\ueb}^2_{L^{\infty}((0,T)\times\R)}\le C(T)+C(T)\norm{\ueb}^2_{L^{\infty}((0,T)\times\R)}.
\end{split}
\end{equation}
Thanks to \eqref{eq:stima-l-2-1},
\begin{equation}
\label{eq:u-4}
\begin{split}
\frac{\gamma}{6}\int_{\R}\ueb^4dx \le&\frac{\gamma}{6}\norm{\ueb}^2_{L^{\infty}((0,T)\times\R)}\norm{\ueb(t,\cdot)}^2_{L^2(\R)}\\
\le&\frac{\gamma}{6}C_{0}e^{2\gamma t}\norm{\ueb}^2_{L^{\infty}((0,T)\times\R)}\le C(T)\norm{\ueb}^2_{L^{\infty}((0,T)\times\R)}.
\end{split}
\end{equation}
\eqref{eq:P123}, \eqref{eq:you4} and \eqref{eq:u-4} give
\begin{equation}
\label{eq:u-5}
\begin{split}
\frac{d}{dt}G_{1}(t) &-2\gamma G_{1}(t) +2\beta\eps\norm{\pxx\ueb(t,\cdot)}^2_{L^2(\R)}\\
\le & C(T)+ C(T)\norm{\ueb}^2_{L^{\infty}((0,T)\times\R)}+ \eps\norm{\ueb}^2_{L^{\infty}((0,T)\times\R)}\norm{\px\ueb(t,\cdot)}^2_{L^2(\R)}\\
&+ \frac{3\eps}{2} \norm{\px\ueb(t,\cdot)}^2_{L^2(\R)}.
\end{split}
\end{equation}
It follows from \eqref{eq:u0eps}, \eqref{eq:stima-l-2-1}, \eqref{eq:def-di-G}, \eqref{eq:u-5} and Gronwall's Lemma that
\begin{equation}
\label{eq:u-6}
\begin{split}
&\beta\norm{\px\ueb(t,\cdot)}^2_{L^2(\R)} +\gamma\norm{\Peb(t,\cdot)}^2_{L^2(\R)} \\
&\qquad + \eps^2\gamma \norm{\px\Peb(t,\cdot)}^2_{L^2(\R)} + 2\beta\eps e^{2\gamma t}\int_{0}^{t}e^{-2 C\gamma s}\norm{\pxx\ueb(s,\cdot)}^2_{L^2(\R)}ds\\
&\quad \le C_{0}e^{2\gamma t} + C(T)e^{2 \gamma t}\int_{0}^{t}e^{-2 \gamma s}ds +e^{2\gamma t}C(T)\norm{\ueb}^2_{L^{\infty}((0,T)\times\R)}\int_{0}^{t}e^{-2 \gamma s}ds\\
&\qquad+ \eps e^{2 \gamma t}\norm{\ueb}^2_{L^{\infty}((0,T)\times\R)}\int_{0}^{t}e^{-2\gamma s}ds \norm{\px\ueb(s,\cdot)}^2_{L^2(\R)}ds\\
&\qquad +\frac{3\eps}{2}e^{2\gamma t}\int_{0}^{t} e^{-2\gamma s}\norm{\px\ueb(s,\cdot)}^2_{L^2(\R)} ds +\frac{1}{12}\int_{\R}\ueb^4dx\\
&\quad\le C(T) +C(T)\norm{\ueb}^2_{L^{\infty}((0,T)\times\R)}+\frac{C_{0}e^{2\gamma t}}{2}\norm{\ueb}^2_{L^{\infty}((0,T)\times\R)}\\
&\qquad 3C_{0}e^{2\gamma t} +\frac{1}{12}\norm{\ueb}^2_{L^{\infty}((0,T)\times\R)}\norm{\ueb(t,\cdot)}^2_{L^2(\R)}\\
&\quad \le C(T) + C(T)\norm{\ueb}^2_{L^{\infty}((0,T)\times\R)} +\frac{C_{0}e^{2\gamma t}}{12}\norm{\ueb}^2_{L^{\infty}((0,T)\times\R)}\\
&\quad \le C(T)\left(1+\norm{\ueb}^2_{L^{\infty}((0,T)\times\R)}\right).
\end{split}
\end{equation}
Following \cite{Cd2, CK, LN}, we begin by observing that, due to \eqref{eq:stima-l-2-1}, \eqref{eq:u-6} and the H\"older inequality,
\begin{align*}
\ueb^2(t,x)&=2\int_{-\infty}^{x}\ueb\px\ueb dy\le 2\int_{\R}\vert\ueb\px\ueb\vert dx\\
&\le\frac{2}{\sqrt{\beta}}\norm{\ueb}_{L^2(\R)}\sqrt{\beta}\norm{\px\ueb(t,\cdot)}_{L^2(\R)}\\
&\le\frac{2}{\sqrt{\beta}}C_{0}e^{\gamma t}\sqrt{ C(T)\left(1+\norm{\ueb}^2_{L^{\infty}((0,T)\times\R)}\right)}\\
&\le \frac{C(T)}{\sqrt{\beta}}\sqrt{\left(1+\norm{\ueb}^2_{L^{\infty}((0,T)\times\R)}\right)},
\end{align*}
that is
\begin{equation}
\label{eq:quarto-grado}
\norm{\ueb}^4_{L^{\infty}((0,T)\times\R)} \le \frac{C(T)}{\beta}\left( 1+\norm{\ueb}^2_{L^{\infty}((0,T)\times\R)}\right).
\end{equation}
Let us show that \eqref{eq:quarto-grado} is verified when
\begin{equation}
\label{eq:sol-eq-quarto}
\norm{\ueb}_{L^{\infty}((0,T)\times\R)}\le \max\left\{1, \left(\frac{2C(T)}{\beta}\right)^{\frac{1}{2}}\right\}\le C(T)\beta^{-\frac{1}{2}}.
\end{equation}
We write $y= \norm{\ueb}_{L^{\infty}((0,T)\times\R)}$. Therefore, it follows from \eqref{eq:quarto-grado} that
\begin{equation}
\label{eq:equat-in-y}
y^4 \le  \frac{C(T)}{\beta}\left( 1+y^2 \right).
\end{equation}
If $y\le 1$, we have \eqref{eq:sol-eq-quarto}.

If $y>1$ and
\begin{equation*}
y>\left(\frac{2C(T)}{\beta}\right)^{\frac{1}{2}},
\end{equation*}
that is
\begin{equation}
\label{eq:potenza}
y^2>\frac{2C(T)}{\beta},  \quad y^4>\frac{2C(T)}{\beta}y^2.
\end{equation}
It follows from \eqref{eq:potenza} that
\begin{equation*}
2y^4>y^4+y^2>\frac{2C(T)}{\beta}(y^2+1).
\end{equation*}
Therefore,
\begin{equation*}
y^4>\frac{C(T)}{\beta}(y^2+1),
\end{equation*}
which is in contradiction with \eqref{eq:equat-in-y}.

Finally, \eqref{eq:ux2} follows from \eqref{eq:l-infty-u} and \eqref{eq:u-6}.
\end{proof}

\begin{lemma}\label{lm:stima-l-6}
Let  $T>0$. Assume \eqref{eq:beta-eps} holds true. There exists a function $C(T)>0$, independent on $\eps$ and $\beta$, such that
\begin{equation}
\label{eq:P-l-infty}
\norm{\Peb}_{L^{\infty}((0,T)\times\R)}\le C(T).
\end{equation}
In particular, we have that
\begin{align}
\label{eq:P-in-l2}
\norm{\Peb(t,\cdot)}_{L^{2}(\R)}\le & C(T),\\
\label{pxP-in-l2-1}
\eps\norm{\px\Peb(t,\cdot)}_{L^{2}(\R)}\le & C(T),\\
\label{eq:u-in-l-6}
\norm{\ueb(t,\cdot)}_{L^{6}(\R)}\le & C(T),\\
\label{eq:eps-px-u-l-2}
\eps\norm{\px\ueb(t,\cdot)}_{L^{2}(\R)}\le & C(T),\\
\label{eq:p001}
\eps e^{t}\int_{0}^{t}\!\!\!\int_{\R}e^{-s}\ueb^4(s,\cdot)(\px\ueb(s,\cdot))^2 dsdx\le &C(T),\\
\label{eq:p002}
\eps^3 e^{t}\int_{0}^{t} e^{-s} \norm{\pxx\ueb(s,\cdot)}^2_{L^2(\R)}ds \le & C(T),
\end{align}
for every $0<t<T$. Moreover,
\begin{align}
\label{eq:00010}
\beta\norm{\px\ueb\pxx\ueb}_{L^{1}((0,T)\times\R)} \le &C(T),\\
\label{eq:defuxx}
\beta^2\int_{0}^{T}\norm{\pxx\ueb(s,\cdot)}^2_{L^2(\R)}ds \le& C(T)\eps.
\end{align}
\end{lemma}
\begin{proof}
Let $0\le t \le T$. Multiplying \eqref{eq:OHepsw} by $\ueb^5 -3\eps^2\pxx\ueb$, we have
\begin{equation}
\label{eq:Ohmp-1}
\begin{split}
(\ueb^5 -3\eps^2\pxx\ueb)\pt\ueb &- \frac{1}{6}(\ueb^5 -3\eps^2\pxx\ueb)\px \ueb^3\\
&-(\ueb^5 -3\eps^2\pxx\ueb)\beta\pxxx\ueb\\
=&\gamma(\ueb^5 -3\eps^2\pxx\ueb)\Peb +\eps(\ueb^5 -3\eps^2\pxx\ueb)\pxx\ueb.
\end{split}
\end{equation}
Since
\begin{align*}
\int_{\R}(\ueb^5 -3\eps^2\pxx\ueb)\pt\ueb dx=&\frac{d}{dt}\left(\frac{1}{6}\norm{\ueb(t,\cdot)}^6_{L^6(\R)}+\frac{3\eps^2}{2}\norm{\px\ueb(t,\cdot)}^2_{L^2(\R)}\right),\\
- \frac{1}{6}\int_{\R}(\ueb^5 -3\eps^2\pxx\ueb)\px \ueb^3 dx=&\frac{3\eps^2}{2}\int_{\R}\ueb^2\px\ueb\pxx\ueb dx,\\
-\beta\int_{\R}(\ueb^5 -3\eps^2\pxx\ueb)\pxxx\ueb dx=& 5\beta\int_{\R}\ueb^4\px\ueb\pxx\ueb dx\\
=& -10\beta\int_{\R}\ueb^3(\px\ueb)^2dx,\\
\eps\int_{\R}(\ueb^5 -3\eps^2\pxx\ueb)\pxx\ueb dx = & -5\eps\int_{\R}\ueb^4(\px\ueb)^2 dx -3\eps^3\norm{\pxx\ueb(t,\cdot)}^2_{L^2(\R)},
 \end{align*}
integrating \eqref{eq:Ohmp-1} over $\R$,
\begin{equation}
\label{eq:Rhs1}
\begin{split}
&\frac{d}{dt}\left(\frac{1}{6}\norm{\ueb(t,\cdot)}^6_{L^6(\R)}+\frac{3\eps^2}{2}\norm{\px\ueb(t,\cdot)}^2_{L^2(\R)}\right)\\
&\qquad + 5\eps\int_{\R}\ueb^4(\px\ueb)^2 dx + 3\eps^3\norm{\pxx\ueb(t\cdot)}^2_{L^2(\R)}\\
&\quad = \gamma \int_{\R}\ueb^5 \Peb dx -3\gamma\eps^2\int_{\R}\pxx\ueb\Peb dx\\
&\qquad -\frac{3\eps^2}{2}\int_{\R}\ueb^2\px\ueb\pxx\ueb dx -10\beta\int_{\R}\ueb^3(\px\ueb)^2 dx.
\end{split}
\end{equation}
Due to \eqref{eq:stima-l-2-1} and the Young inequality,
\begin{equation}
\label{eq:you8}
\begin{split}
&\gamma \left \vert\int_{\R}\ueb^5 \Peb dx\right\vert \le \int_{\R}\left\vert \sqrt{6}\gamma\ueb^2\Peb\right\vert \left\vert\frac{\ueb^3}{\sqrt{6}}\right\vert dx\\
&\quad \le 3\gamma^2\int_{\R}\Peb^2\ueb^4 dx + \frac{1}{12}\norm{\ueb(t,\cdot)}^6_{L^{6}(\R)}\\
&\quad \le \int_{\R}\left\vert3\sqrt{6}\gamma^2\Peb^2\ueb\right\vert\left\vert\frac{\ueb^3}{\sqrt{6}}\right\vert dx + \frac{1}{12}\norm{\ueb(t,\cdot)}^6_{L^{6}(\R)}\\
&\quad \le 27\gamma^4\int_{\R}\Peb^4\ueb^2 dx +  \frac{1}{6}\norm{\ueb(t,\cdot)}^6_{L^{6}(\R)}\\
&\quad \le 27\gamma^4\norm{\Peb}^4_{L^{\infty}((0,T)\times\R)}\norm{\ueb(t,\cdot)}^2_{L^{2}(\R)}+  \frac{1}{6}\norm{\ueb(t,\cdot)}^6_{L^{6}(\R)}\\
&\quad \le 27\gamma^4C_{0}e^{2\gamma t}\norm{\Peb}^4_{L^{\infty}((0,T)\times\R)}+  \frac{1}{6}\norm{\ueb(t,\cdot)}^6_{L^{6}(\R)}\\
&\quad \le C(T)\norm{\Peb}^4_{L^{\infty}((0,T)\times\R)}+  \frac{1}{6}\norm{\ueb(t,\cdot)}^6_{L^{6}(\R)}.
\end{split}
\end{equation}
Since $0<\eps<1$, it follows from \eqref{eq:P-pxP-intfy-1}, \eqref{eq:10021-1} and the Young inequality that
\begin{equation}
\label{eq:you9}
\begin{split}
&-3\gamma\eps^2\int_{\R}\pxx\ueb\Peb dx= 3\gamma\eps^2\int_{\R}\px\ueb\px\Peb dx \\
&\quad\le 3\gamma\eps^2\left\vert \int_{\R}\px\ueb\px\Peb dx\right\vert \le 3\eps^2 \int_{\R}\vert \px\ueb\vert \left\vert\gamma\px\Peb\right\vert dx\\
&\quad \le \frac{3\eps^2}{2} \norm{\px\ueb(t,\cdot)}^2_{L^{2}(\R)} +\frac{3\gamma^2\eps^2}{2} \norm{\px\Peb(t,\cdot)}^2_{L^{2}(\R)}\\
&\quad \le \frac{3\eps^2}{2} \norm{\px\ueb(t,\cdot)}^2_{L^{2}(\R)} +\frac{3\gamma^2\eps^2}{2}C_{0}e^{2\gamma t}\\
&\quad \le \frac{3\eps^2}{2} \norm{\px\ueb(t,\cdot)}^2_{L^{2}(\R)}+\frac{3\gamma^2}{2}C_{0}e^{2\gamma t}\\
&\quad \le \frac{3\eps}{2} \norm{\px\ueb(t,\cdot)}^2_{L^{2}(\R)}+C(T).
\end{split}
\end{equation}
Thanks to \eqref{eq:beta-eps}, \eqref{eq:l-infty-u} and the Young inequality,
\begin{equation}
\begin{split}
&10\beta\left\vert\int_{\R}\ueb^3(\px\ueb)^2 dx\right\vert\le 10\beta\int_{\R}\vert\ueb\vert^3 (\px\ueb)^2 dx \\
&\quad \le 10\beta \norm{\ueb}_{L^{\infty}((0,T)\times\R)}\int_{\R}\ueb^2(\px\ueb)^2 dx\le C(T)\beta^{\frac{1}{2}}\int_{\R}\ueb^2(\px\ueb)^2 dx\\
&\quad\le\eps\int_{\R}\vert C_{0}C(T)\px\ueb\vert\ueb^2\vert\px\ueb\vert dx\\
&\quad\le \eps C(T)\norm{\px\ueb(t,\cdot)}^2_{L^2(\R)} + \frac{\eps}{2}\int_{\R}\ueb^4 (\px\ueb)^2 dx.
\end{split}
\end{equation}
For the Young inequality,
\begin{equation}
\label{eq:you10}
\begin{split}
&\frac{3\eps^2}{2}\left\vert\int_{\R}\ueb^2\px\ueb\pxx\ueb dx\right\vert \le \frac{3}{2}\int_{\R}\left\vert\eps^{\frac{1}{2}}\ueb^2\px\ueb\right\vert\left\vert\eps^{\frac{3}{2}}\pxx\ueb\right\vert dx\\
&\quad \le\frac{3\eps}{4}\int_{\R}\ueb^4(\px\ueb)^2dx +\frac{3\eps^3}{4}\norm{\pxx\ueb(t,\cdot)}^2_{L^2(\R)}.
\end{split}
\end{equation}
It follows from \eqref{eq:Rhs1}, \eqref{eq:you8}, \eqref{eq:you9} and \eqref{eq:you10} that
\begin{align*}
&\frac{d}{dt}G_{2}(t)+\frac{15\eps}{4}\int_{\R}\ueb^4(\px\ueb)^2 dx +  \frac{9\eps^3}{4}\norm{\pxx\ueb(t,\cdot)}^2_{L^2(\R)}\\
&\quad \le G_{2}(t) +C(T) + C(T)\norm{\Peb}^4_{L^{\infty}((0,T)\times\R)}+\eps C(T)\norm{\px\ueb(t,\cdot)}^2_{L^2(\R)},
\end{align*}
that is
\begin{equation}
\label{eq:gro1}
\begin{split}
&\frac{d}{dt}G_{2}(t) - G_{2}(t)+\frac{15\eps}{4}\int_{\R}\ueb^4(\px\ueb)^2 dx +  \frac{9\eps^3}{4}\norm{\pxx\ueb(t,\cdot)}^2_{L^2(\R)}\\
&\quad \le C(T) + C(T)\norm{\Peb}^4_{L^{\infty}((0,T)\times\R)}+\eps C(T)\norm{\px\ueb(t,\cdot)}^2_{L^2(\R)},
\end{split}
\end{equation}
where
\begin{equation}
\label{eq:def-di-G1}
G_{2}(t)= \frac{1}{6}\norm{\ueb(t,\cdot)}^6_{L^6(\R)}+\frac{3\eps^2}{2}\norm{\px\ueb(t,\cdot)}^2_{L^2(\R)}.
\end{equation}
The Gronwall Lemma and \eqref{eq:u0eps} give
\begin{equation}
\label{eq:gro2}
\begin{split}
G_{2}(t)&+ \frac{15\eps}{4}e^{t}\int_{0}^{t}\!\!\!\int_{\R}e^{-s}\ueb^4(s,\cdot)(\px\ueb(s,\cdot))^2 dsdx \\
&+  \frac{9\eps^3}{4}e^{t}\int_{0}^{t} e^{-s} \norm{\pxx\ueb(s,\cdot)}^2_{L^2(\R)}ds\\
\le& C_{0}e^{t} + C(T)e^{t}\int_{0}^{t} e^{-s} ds  + C(T)e^{t}\norm{\Peb}^4_{L^{\infty}((0,T)\times\R)}\int_{0}^{t}e^{-s} ds\\
&+\eps C(T)e^{t}\int_{0}^{t}e^{-s} \norm{\px\ueb(s,\cdot)}^2_{L^2(\R)}ds\\
\le& C(T) + C(T)\norm{\Peb}^4_{L^{\infty}((0,T)\times\R)}+\eps C(T)\int_{0}^{t}\norm{\px\ueb(s,\cdot)}^2_{L^2(\R)}ds.
\end{split}
\end{equation}
Due to \eqref{eq:stima-l-2-1},
\begin{equation}
\label{eq:ux-45}
\begin{split}
\eps\int_{0}^{t}\norm{\px\ueb(s,\cdot)}^2_{L^2(\R)}ds\le& \eps e^{2\gamma t}\int_{0}^{t}e^{-2\gamma s}\norm{\px\ueb(s,\cdot)}^2_{L^2(\R)}ds\\
\le& C_{0}e^{2\gamma t} \le C(T).
\end{split}
\end{equation}
Therefore, for \eqref{eq:def-di-G1}, \eqref{eq:gro2} and \eqref{eq:ux-45} we have
\begin{equation}
\label{eq:stima-u-6}
\begin{split}
\frac{1}{6}\norm{\ueb(t,\cdot)}^6_{L^6(\R)}&+\frac{3\eps^2}{2}\norm{\px\ueb(t,\cdot)}^2_{L^2(\R)}\\
&+ \frac{15\eps}{4}e^{t}\int_{0}^{t}\!\!\!\int_{\R}e^{-s}\ueb^4(s,\cdot)(\px\ueb(s,\cdot))^2 dsdx\\
&+  \frac{9\eps^3}{4}e^{t}\int_{0}^{t} e^{-s} \norm{\pxx\ueb(s,\cdot)}^2_{L^2(\R)}ds\\
\le & C(T) + C(T)\norm{\Peb}^4_{L^{\infty}((0,T)\times\R)}.
\end{split}
\end{equation}
Dividing \eqref{eq:P16} by $\gamma$, we get
\begin{equation}
\label{eq:equat-p-12}
\begin{split}
&\frac{d}{dt}\left(\norm{\Peb(t,\cdot)}^2_{L^2(\R)} + \eps^2\norm{\px\Peb(t,\cdot)}^2_{L^2(\R)}\right)\\
&\quad = -2\eps\gamma\int_{\R}\Feb\px\Peb dx +\frac{1}{3}\int_{\R}\ueb^3\Peb dx\\
&\qquad -\frac{\eps}{3}\int_{\R}\ueb^3\px\Peb dx +2\beta\int_{\R}\pxx\ueb\Peb dx\\
&\qquad -2\beta\eps\int_{\R}\pxx\ueb\px\Peb dx +2\eps\int_{\R}\px\ueb\Peb dx\\
&\qquad -2\eps^2\int_{\R}\px\ueb\px\Peb dx.
\end{split}
\end{equation}
Since $0<\eps<1$, due to \eqref{eq:P-pxP-intfy-1} and \eqref{eq:F1},
\begin{equation}
\label{eq:FpxP-1}
-2\eps\gamma\int_{\R}\Feb\px\Peb dx = 2\eps\gamma\int_{\R}\px\Fe\Pe dx = 2\eps \gamma\int_{\R}\Peb^2 dx \le 2\gamma \norm{\Peb(t,\cdot)}^2_{L^2(\R)}.
\end{equation}
It follows from \eqref{eq:beta-eps}, \eqref{eq:P-pxP-intfy-1} and the Young inequality that
\begin{equation}
\label{eq:you15}
\begin{split}
2\beta\int_{\R}\pxx\ueb\Peb dx=& -2\beta\int_{\R}\px\ueb\px\Peb dx\\
\le& 2\beta\left\vert\int_{\R}\px\ueb\px\Peb dx\right\vert\\
\le& \eps^2 \int_{\R} \left\vert\frac{C_{0}\px\ueb}{\sqrt{\gamma}}\right\vert\left\vert 2\sqrt{\gamma}\px\Peb\right\vert dx\\
\le& \frac{C_{0}\eps^2}{2\gamma}\norm{\px\ueb(t,\cdot)}^2_{L^2(\R)} + 2\gamma\eps^2 \norm{\px\Peb(t,\cdot)}^2_{L^2(\R)}\\
\le& \frac{C_{0}\eps}{2\gamma}\norm{\px\ueb(t,\cdot)}^2_{L^2(\R)} + 2\gamma\eps^2 \norm{\px\Peb(t,\cdot)}^2_{L^2(\R)}.
\end{split}
\end{equation}
Thanks to \eqref{eq:beta-eps}, \eqref{eq:P-pxP-intfy-1}, \eqref{eq:10021-1} and the Young inequality,
\begin{equation}
\label{eq:you16}
\begin{split}
-2\beta\eps\int_{\R}\pxx\ueb\px\Peb dx =& 2\beta\eps\int_{\R}\px\ueb\pxx\Peb dx\\
\le & 2\beta\eps \left\vert \int_{\R}\px\ueb\pxx\Peb dx\right\vert\\
\le &2\eps^3\int_{\R}\left\vert C_{0}\px\ueb\right\vert\vert\pxx\Peb\vert dx\\
\le &\eps^3C_{0}\norm{\px\ueb(t,\cdot)}^2_{L^2(\R)}+\eps^3\norm{\pxx\Peb(t,\cdot)}^2_{L^2(\R)}\\
\le &\eps C_{0}\norm{\px\ueb(t,\cdot)}^2_{L^2(\R)}+C_{0}e^{2\gamma t}\\
\le & \eps C_{0}\norm{\px\ueb(t,\cdot)}^2_{L^2(\R)}+C(T).
\end{split}
\end{equation}
It follows from \eqref{eq:P-pxP-intfy-1}, \eqref{eq:stima-l-2-1}, \eqref{eq:10021-1} and the Young inequality that
\begin{equation}
\label{eq:you18}
\begin{split}
&2\eps\int_{\R}\px\ueb\Peb dx -2 \eps^2\int_{\R}\px\ueb\px\Peb dx\\
&\quad  = -2\eps\int_{\R}\ueb\px\Peb dx  + 2\eps^2\int_{\R}\ueb\pxx\Peb dx\\
&\quad \le\left \vert -2\eps\int_{\R}\ueb\px\Peb dx +2 \eps^2\int_{\R}\ueb\pxx\Peb dx\right\vert\\
&\quad \le 2\eps \int_{\R}\vert\ueb\vert \vert \px\Peb \vert dx + 2\eps^2\int_{\R}\vert\ueb\vert\vert\pxx\Peb\vert dx\\
&\quad \le 2\norm{\ueb(t,\cdot)}^2_{L^2(\R)} +\norm{\px\Peb(t,\cdot)}^2_{L^2(\R)}+\eps \norm{\pxx\Peb(t,\cdot)}^2_{L^2(\R)}\\
&\quad \le 3C_{0}e^{2\gamma t}\le C(T).
\end{split}
\end{equation}
Therefore, \eqref{eq:equat-p-12}, \eqref{eq:FpxP-1}, \eqref{eq:you15}, \eqref{eq:you16} and \eqref{eq:you18} give
\begin{equation}
\label{eq:equat-p-13}
\begin{split}
&\frac{d}{dt}\left(\norm{\Peb(t,\cdot)}^2_{L^2(\R)} + \eps^2\norm{\px\Peb(t,\cdot)}^2_{L^2(\R)}\right)\\
&\qquad -2\gamma\left(\norm{\Peb(t,\cdot)}^2_{L^2(\R)} +\eps^2\norm{\px\Peb(t,\cdot)}^2_{L^2(\R)}\right)\\
&\quad \le C(T) + \left(\frac{1}{\gamma}+1\right)C_{0}\eps\norm{\px\ueb(t,\cdot)}^2_{L^2(\R)}\\
&\qquad +\frac{1}{3}\int_{\R}\vert\ueb\vert^3 \vert \Peb \vert dx +\frac{\eps}{3}\int_{\R} \vert\ueb\vert^3 \vert\px\Peb\vert dx.
\end{split}
\end{equation}
Due to \eqref{eq:stima-l-2-1}, \eqref{eq:10021-1} \eqref{eq:stima-u-6} and the Young inequality,
\begin{equation}
\label{eq:you19}
\begin{split}
&\frac{1}{3}\int_{\R}\vert\ueb\vert^3 \vert \Peb \vert dx +\frac{\eps}{3}\int_{\R} \vert\ueb\vert^3 \vert\px\Peb\vert dx\\
&\quad \le \frac{1}{6}\int_{\R}\ueb^2 \Peb^2 dx + \frac{1}{6}\int_{\R} \ueb^4 dx\\
&\qquad +\frac{\eps^2}{6}\int_{\R}\ueb^2(\px\Peb)^2 dx + \frac{1}{6}\int_{\R} \ueb^4 dx\\
&\quad \le \frac{1}{6}\norm{\Peb}^2_{L^{\infty}((0,T)\times\R)}\norm{\ueb(t,\cdot)}^2_{L^2(\R)}+\frac{\eps^2}{6}\norm{\px\Peb(t,\cdot)}^2_{L^{\infty}(\R)}\norm{\ueb(t,\cdot)}^2_{L^2(\R)}\\
&\qquad +\frac{1}{3}\int_{\R}\left\vert\frac{\ueb}{\sqrt{A(T)}}\right\vert \left\vert \sqrt{A(T)} \ueb^3\right \vert dx\\
&\quad \le \frac{C_{0}e^{2\gamma t}}{6}\norm{\Peb}^2_{L^{\infty}((0,T)\times\R)}+\frac{C_{0} e^{4\gamma t}}{6} +\frac{1}{6A(T)}\norm{\ueb(t,\cdot)}^2_{L^2(\R)}\\    &\qquad + \frac{A(T)}{6}\norm{\ueb(t,\cdot)}^6_{L^6(\R)}\\
&\quad \le C(T)\norm{\Peb}^2_{L^{\infty}((0,T)\times\R)} +C(T)+\frac{C_{0}e^{2\gamma t}}{6A(T)} +C(T)A(T)\\
&\qquad + C(T)A(T)\norm{\Peb}^4_{L^{\infty}((0,T)\times\R)}\\
&\quad \le C(T) +\frac{C(T)}{A(T)} +  C(T)\norm{\Peb}^2_{L^{\infty}((0,T)\times\R)} + C(T)A(T) \norm{\Peb}^4_{L^{\infty}((0,T)\times\R)},
\end{split}
\end{equation}
where $A(T)$ is a positive constant which will be specified later.

Thus, \eqref{eq:equat-p-13} and \eqref{eq:you19} give
\begin{align*}
&\frac{d}{dt}\left(\norm{\Peb(t,\cdot)}^2_{L^2(\R)} + \eps^2\norm{\px\Peb(t,\cdot)}^2_{L^2(\R)}\right)\\
&\qquad -2\gamma\left(\norm{\Peb(t,\cdot)}^2_{L^2(\R)} + \eps^2\norm{\px\Peb(t,\cdot)}^2_{L^2(\R)}\right)\\
&\quad \le C(T) + \frac{C(T)}{A(T)}  + \left(\frac{1}{\gamma}+1\right)C_{0} \eps\norm{\px\ueb(t,\cdot)}^2_{L^2(\R)}\\
&\qquad  +  C(T)\norm{\Peb}^2_{L^{\infty}((0,T)\times\R)}+ C(T)A(T) \norm{\Peb}^4_{L^{\infty}((0,T)\times\R)}.
\end{align*}
It follows from \eqref{eq:u0eps}, \eqref{eq:stima-l-2-1} and the Gronwall Lemma that
\begin{align*}
&\norm{\Peb(t,\cdot)}^2_{L^2(\R)} + \eps^2\norm{\px\Peb(t,\cdot)}^2_{L^2(\R)}\\
&\quad \le C_{0}e^{2\gamma t} + C(T)e^{2\gamma t}\int_{0}^{t}e^{-2\gamma s} ds  + \frac{C(T)}{A(T)}e^{2\gamma t}\int_{0}^{t}e^{-2\gamma s} ds\\
&\qquad + \left(\frac{1}{\gamma}+1\right)C_{0}e^{2\gamma t}\eps \int_{0}^{t}e^{-2\gamma s} \norm{\px\ueb(s,\cdot)}^2_{L^2(\R)}\\
&\qquad +C(T)e^{2\gamma t}\norm{\Peb}^2_{L^{\infty}((0,T)\times\R)}\int_{0}^{t} e^{-2\gamma s} ds\\
&\qquad + C(T)A(T) e^{2\gamma t}\norm{\Peb}^4_{L^{\infty}((0,T)\times\R)}\int_{0}^{t} e^{-2\gamma s} ds\\
&\quad \le C(T) +\frac{C(T)}{A(T)}+\left(\frac{1}{\gamma}+1\right)\frac{C_{0}e^{2\gamma t}}{2}+C(T)\norm{\Peb}^2_{L^{\infty}((0,T)\times\R)}\\
&\qquad +C(T) A(T) \norm{\Peb}^4_{L^{\infty}((0,T)\times\R)},
\end{align*}
that is
\begin{equation}
\label{eq:Pl2}
\begin{split}
&\norm{\Peb(t,\cdot)}^2_{L^2(\R)} + \eps^2\norm{\px\Peb(t,\cdot)}^2_{L^2(\R)}\\
&\quad \le C(T)\left ( 1+\frac{1}{A(T)} +\norm{\Peb}^2_{L^{\infty}((0,T)\times\R)}+  A(T) \norm{\Peb}^4_{L^{\infty}((0,T)\times\R)}\right).
\end{split}
\end{equation}
We prove \eqref{eq:P-l-infty}. Due to  the H\"older inequality,
\begin{equation*}
\Peb^2(t,x)\le 2\int_{\R}\vert\Peb\px\Peb\vert dx \le \norm{\Peb(t,\cdot)}_{L^2(\R)}\norm{\px\Peb(t,\cdot)}_{L^2(\R)},
\end{equation*}
that is
\begin{equation*}
\Peb^4(t,x)\le 4 \norm{\Peb(t,\cdot)}^2_{L^2(\R)}\norm{\px\Peb(t,\cdot)}^2_{L^2(\R)}.
\end{equation*}
For \eqref{eq:10021-1} and  \eqref{eq:Pl2},
\begin{equation*}
\Peb^4(t,x)\le C(T)\left ( 1+\frac{1}{A(T)} +\norm{\Peb}^2_{L^{\infty}((0,T)\times\R)}+  A(T) \norm{\Peb}^4_{L^{\infty}((0,T)\times\R)}\right).
\end{equation*}
Therefore,
\begin{equation*}
\label{eq:P-quarto}
\left(1-C(T)A(T)\right)\norm{\Peb}^4_{L^{\infty}((0,T)\times\R)}-C(T)\norm{\Peb}^2_{L^{\infty}((0,T)\times\R)}-C(T) -\frac{C(T)}{A(T)}\le 0.
\end{equation*}
Choosing
\begin{equation}
\label{eq:def-di-A}
A(T)=\frac{1}{2C(T)},
\end{equation}
we have that
\begin{equation*}
\frac{1}{2}\norm{\Peb}^4_{L^{\infty}((0,T)\times\R)}-C(T)\norm{\Peb}^2_{L^{\infty}((0,T)\times\R)} -C(T) \le 0,
\end{equation*}
which gives \eqref{eq:P-l-infty}.
Therefore, \eqref{eq:P-l-infty} and \eqref{eq:def-di-A} give \eqref{eq:P-in-l2} and \eqref{pxP-in-l2-1}, while \eqref{eq:u-in-l-6}, \eqref{eq:eps-px-u-l-2}, \eqref{eq:p001} and \eqref{eq:p002} follow from \eqref{eq:P-l-infty}.

We show that \eqref{eq:defuxx} holds true. We begin by observing that, for \eqref{eq:p002},
\begin{equation*}
\eps^3\int_{0}^{t}\norm{\pxx\ueb(s,\cdot)}^2_{L^2(\R)}ds \le \eps^3 e^{t}\int_{0}^{t}e^{-s}\norm{\pxx\ueb(s,\cdot)}^2_{L^2(\R)}ds\le C(T).
\end{equation*}
Therefore,
\begin{equation}
\label{eq:0005}
\eps^3\int_{0}^{T}\norm{\pxx\ueb(t,\cdot)}^2_{L^2(\R)}dt \le C(T).
\end{equation}
Moreover, for \eqref{eq:ux-45},
\begin{equation}
\label{eq:0006}
\eps\int_{0}^{T}\norm{\px\ueb(t,\cdot)}^2_{L^2(\R)}dt \le C(T).
\end{equation}
Thus, thanks to  \eqref{eq:beta-eps}, \eqref{eq:0005}, \eqref{eq:0006} and H\"older inequality,
\begin{align*}
&\beta\int_{0}^{T}\!\!\!\int_{\R}\vert\px\ueb\pxx\ueb\vert dsdx =\frac{\beta}{\eps^2}\int_{0}^{T}\!\!\!\int_{\R}\eps^{\frac{1}{2}}\vert\px\ueb\vert\eps^{\frac{3}{2}}\vert\pxx\ueb\vert dx\\
&\quad \le \frac{\beta}{\eps^2} \left(\eps \int_{0}^{T}\!\!\!\int_{\R}(\px\ueb)^2 dsdx\right)^{\frac{1}{2}}\left(\eps^3 \int_{0}^{T}\!\!\!\int_{\R}(\pxx\ueb)^2 dsdx\right)^{\frac{1}{2}}\\
&\quad \le C(T)\frac{\beta}{\eps^2}\le C(T),
\end{align*}
that is \eqref{eq:00010}.

Finally, we prove \eqref{eq:defuxx}. Due to \eqref{eq:beta-eps} and \eqref{eq:0005}, we have
\begin{align*}
\beta^2\int_{0}^{T}\norm{\pxx\ueb(s,\cdot)}^2_{L^2(\R)}ds \le C_{0}^2\eps^4\int_{0}^{T}\norm{\pxx\ueb(s,\cdot)}^2_{L^2(\R)}ds\le C(T)\eps,
\end{align*}
which gives \eqref{eq:defuxx}.
\end{proof}

\section{Proof of Theorem \ref{th:main}}
\label{sec:theor}
In this section, we prove Theorem \ref{th:main}. The following technical lemma is needed  \cite{Murat:Hneg}.
\begin{lemma}
\label{lm:1}
Let $\Omega$ be a bounded open subset of $
\R^2$. Suppose that the sequence $\{\mathcal
L_{n}\}_{n\in\mathbb{N}}$ of distributions is bounded in
$W^{-1,\infty}(\Omega)$. Suppose also that
\begin{equation*}
\mathcal L_{n}=\mathcal L_{1,n}+\mathcal L_{2,n},
\end{equation*}
where $\{\mathcal L_{1,n}\}_{n\in\mathbb{N}}$ lies in a
compact subset of $H^{-1}_{loc}(\Omega)$ and
$\{\mathcal L_{2,n}\}_{n\in\mathbb{N}}$ lies in a
bounded subset of $\mathcal{M}_{loc}(\Omega)$. Then $\{\mathcal
L_{n}\}_{n\in\mathbb{N}}$ lies in a compact subset of $H^{-1}_{loc}(\Omega)$.
\end{lemma}
Moreover, we consider the following definition.
\begin{definition}
A pair of functions $(\eta, q)$ is called an  entropy--entropy flux pair if $\eta :\R\to\R$ is a $C^2$ function and $q :\R\to\R$ is defined by
\begin{equation*}
q(u)=-\int_{0}^{u} \frac{\xi^2}{2}\eta'(\xi) d\xi.
\end{equation*}
An entropy-entropy flux pair $(\eta,\, q)$ is called  convex/compactly supported if, in addition, $\eta$ is convex/compactly supported.
\end{definition}
We begin by proving the following result.
\begin{lemma}\label{lm:dist-solution}
Assume that \eqref{eq:assinit}, \eqref{eq:def-di-P0}, \eqref{eq:L-2P0}, \eqref{eq:u0eps}, and \eqref{eq:beta-eps} hold. Then for any compactly supported entropy--entropy flux pair $(\eta,\, q)$, there exist two sequences $\{\eps_{n}\}_{n\in\N}$, $\{\beta_{n}\}_{n\in\N}$, with $\eps_n, \beta_n \to 0$, and two limit functions
\begin{align*}
&u\in L^{\infty}(0,T; L^2(\R)\cap L^6(\R)),\\
&P\in L^{\infty}((0,T)\times\R)\cap L^{2}((0,T)\times\R),
\end{align*}
 such that
\begin{align}
\label{eq:con1}
& u_{\eps_n, \beta_n}\to u \quad  \textrm{in} \quad  L^{p}_{loc}((0,\infty)\times\R),\quad \textrm{for each} \quad 1\le p <6,\\
\label{eq:con2}
&P_{\eps_n, \beta_n}\to P \quad \textrm{in} \quad L^{\infty}((0,T)\times\R)\cap L^{2}((0,T)\times\R),
\end{align}
and $u$ is a distributional solution of \eqref{eq:SHP}.
\end{lemma}
\begin{proof}

Let us consider a compactly supported entropy--entropy flux pair $(\eta, q)$. Multiplying \eqref{eq:OHepsw} by $\eta'(\ueb)$, we have
\begin{align*}
\pt\eta(\ueb) + \px q(\ueb) =&\eps \eta'(\ueb) \pxx\ueb + \beta \eta'(\ueb) \pxxx\ueb + \gamma \eta'(\ueb) \Peb\\
=& I_{1,\,\eps,\,\beta}+I_{2,\,\eps,\,\beta}+ I_{3,\,\eps,\,\beta} + I_{4,\,\eps,\,\beta}+I_{5,\,\eps,\,\beta},
\end{align*}
where
\begin{equation}
\begin{split}
\label{eq:12000}
I_{1,\,\eps,\,\beta}&=\px(\eps\eta'(\ueb)\px\ueb),\\
I_{2,\,\eps,\,\beta}&= -\eps\eta''(\ueb)(\px\ueb)^2,\\
I_{3,\,\eps,\,\beta}&= \px(\beta\eta'(\ueb)\pxx\ueb),\\
I_{4,\,\eps,\,\beta}&= -\beta\eta''(\ueb)\px\ueb\pxx\ueb,\\
I_{5,\,\eps,\,\beta}&=\gamma\eta'(\ueb) \Peb.
\end{split}
\end{equation}
Arguing as \cite[Lemma $3.2$]{Cd2}, we have that  $I_{1,\,\eps,\,\beta}\to0$ in $H^{-1}((0,T) \times\R)$, $\{I_{2,\,\eps,\,\beta}\}_{\eps,\beta >0}$ is bounded in $L^1((0,T)\times\R)$, $I_{3,\,\eps,\,\beta}\to0$ in $H^{-1}((0,T) \times\R)$, $\{I_{4,\,\eps,\,\beta}\}_{\eps,\beta >0}$ is bounded in $L^1((0,T)\times\R)$,  and $\{I_{5,\,\eps,\,\beta}\}_{\eps,\beta >0}$ is bounded in $L^1_{loc}((0,\infty)\times\R)$.
Therefore, Lemma \ref{lm:1} and the $L^p$ compensated compactness of \cite{SC} give \eqref{eq:con1}.
\eqref{eq:con2} follows from Lemma \ref{lm:stima-l-6}.

We conclude by proving that $u$ is a distributional solution of \eqref{eq:SHP}.
Let $ \phi\in C^{\infty}(\R^2)$ be a test function with
compact support. We have to prove that
\begin{equation}
\label{eq:k1}
\int_{0}^{\infty}\!\!\!\!\!\int_{\R}\left(u\pt\phi-\frac{u^3}{6}\px\phi\right)dtdx - \gamma\int_{0}^{\infty}\!\!\!\!\!\int_{\R} P\phi dtdx +\int_{\R}u_{0}(x)\phi(0,x)dx=0.
\end{equation}
We have that
\begin{align*}
\int_{0}^{\infty}\!\!\!&\!\!\int_{\R}\left(u_{\eps_{n}, \beta_{n}}\pt\phi-\frac{u^3_{\eps_n, \beta_{n}}}{6}\px\phi\right)dtdx  - \gamma\int_{0}^{\infty}\!\!\!\!\!\int_{\R} P_{\eps_{n},\beta_{n}}\phi dtdx +\int_{\R}u_{0,\eps_n,\beta_n}(x)\phi(0,x)dx\\
&+\eps_{n}\int_{0}^{\infty}\!\!\!\!\!\int_{\R}u_{\eps_{n},\beta_{n}}\pxx\phi dtdx + \eps_n\int_{0}^{\infty}u_{0,\eps_{n},\beta_{n}}(x)\pxx\phi(0,x)dx\\
&- \beta_n\int_{0}^{\infty}\!\!\!\!\int_{\R}u_{\eps_n,\beta_n}\pxxx\phi dt dx - \beta_n\int_{0}^{\infty}u_{0,\eps_n,\beta_n}(x)\pxxx\phi(0,x)dx=0.
\end{align*}
Therefore, \eqref{eq:k1} follows from \eqref{eq:con1}, and \eqref{eq:con2}.
\end{proof}
\begin{lemma}
\label{lm:entropy-solution}
Assume that \eqref{eq:assinit}, \eqref{eq:def-di-P0}, \eqref{eq:L-2P0}, \eqref{eq:u0eps}, and \eqref{eq:beta-eps-1} hold. Then,
\begin{align}
\label{eq:con3}
& u_{\eps_n, \beta_n}\to u \quad  \textrm{in} \quad  L^{p}_{loc}((0,\infty)\times\R),\quad \textrm{for each} \quad 1\le p <6,\\
\label{eq:con4}
&P_{\eps_n, \beta_n}\to P \quad \textrm{in} \quad L^{p}_{loc}(0,\infty;W^{1,\infty}(\R)\cap H^{1}(\R)), \quad \textrm{for each} \quad 1\le p <6,
\end{align}
where $u$ is  the unique entropy solution of \eqref{eq:SHP}.
\end{lemma}
\begin{proof}
Let us consider a compactly supported entropy--entropy flux pair $(\eta, q)$. Multiplying \eqref{eq:OHepsw} by $\eta'(\ueb)$, we obtain that
\begin{align*}
\pt\eta(\ueb) + \px q(\ueb) =&\eps \eta'(\ueb) \pxx\ueb + \beta \eta'(\ueb) \pxxx\ueb + \gamma \eta'(\ueb) \Peb\\
=& I_{1,\,\eps,\,\beta}+I_{2,\,\eps,\,\beta}+ I_{3,\,\eps,\,\beta} + I_{4,\,\eps,\,\beta}+I_{5,\,\eps,\,\beta},
\end{align*}
where $I_{1,\,\eps,\,\beta},\,I_{2,\,\eps,\,\beta},\, I_{3,\,\eps,\,\beta} ,\, I_{4,\,\eps,\,\beta}$ and $I_{5,\,\eps,\,\beta}$ are defined in \eqref{eq:12000}.

Arguing as \cite[Lemma $3.3$]{Cd2}, we obtain that $I_{1,\,\eps,\,\beta}\to0$ in $H^{-1}((0,T) \times\R)$, $\{I_{2,\,\eps,\,\beta}\}_{\eps,\beta >0}$ is bounded in $L^1((0,T)\times\R)$, $I_{3,\,\eps,\,\beta}\to0$ in $H^{-1}((0,T) \times\R)$, $I_{4,\,\eps,\,\beta}\to 0$ in $L^1((0,T)\times\R)$,  and $\{I_{5,\,\eps,\,\beta}\}_{\eps,\beta>0}$ is bounded in $L^1_{loc}((0,\infty)\times\R)$.
Therefore, Lemma \ref{lm:1} gives \eqref{eq:con3}.

We prove \eqref{eq:con4}. We begin by observing that, integrating the second equation of \eqref{eq:OHepsw} on $(0,x)$, we have
\begin{equation}
\label{eq:P-u-1}
\Peb(t,x)=\int_{0}^{x}\ueb(t,y)dy +\eps\px\Peb(t,x) -\eps\px\Peb(t,0).
\end{equation}
Let us show that
\begin{equation}
\label{eq:px-1-1}
\textrm{$\eps\px\Pe\to 0$ in $L^{\infty}((0,T)\times\R)$, $T>0$.}
\end{equation}
It follows from \eqref{eq:10021-1} that
\begin{equation*}
\eps\norm{\px\Pe}_{L^{\infty}((0,T)\times\R)}\leq \sqrt{\eps} e^{\gamma T}\norm{u_0}^2_{L^2(\R)}=\sqrt{\eps}C(T)\to 0,
\end{equation*}
that is \eqref{eq:px-1-1}.\\
We claim that
\begin{equation}
\label{eq:px-0-1}
\textrm{$\eps\px\Pe(\cdot,0)\to 0$ in $L^{\infty}(0,T)$, $T>0$.}
\end{equation}
Again by  \eqref{eq:10021-1}, we have that
\begin{equation*}
\eps\norm{\px\Pe(\cdot,0)}_{L^{\infty}(0,T)}\leq \sqrt{\eps} e^{\gamma T}\norm{u_0}^2_{L^2(\R)}=\sqrt{\eps}C(T)\to 0,
\end{equation*}
that is \eqref{eq:px-0-1}.
Therefore, \eqref{eq:con4} follows from \eqref{eq:con3}, \eqref{eq:P-u-1}, \eqref{eq:px-1-1}, \eqref{eq:px-0-1} and  the H\"older inequality.

We conclude by proving that $u$ is the unique entropy solution of \eqref{eq:SHP}.
Let us consider a compactly supported entropy--entropy flux pair $(\eta, q)$, and $\phi\in C^{\infty}_{c}((0,\infty)\times\R)$ a non--negative function. We have to prove that
\begin{equation}
\label{eq:u-entropy-solution}
\int_{0}^{\infty}\!\!\!\!\!\int_{\R}(\pt\eta(u)+ \px q(u))\phi dtdx - \gamma\int_{0}^{\infty}\!\!\!\!\!\int_{\R}\Peb\eta'(u)\phi dtdx \le0.
\end{equation}
We have that
\begin{align*}
&\int_{0}^{\infty}\!\!\!\!\!\int_{\R}(\px\eta(u_{\eps_{n},\,\beta_{n}})+\px q(u_{\eps_{n},\,\beta_{n}}))\phi dtdx - \gamma\int_{0}^{\infty}\!\!\!\!\!\int_{\R}P_{\eps_{n},\,\beta_{n}}\eta'(u_{\eps_{n},\,\beta_{n}})\phi dtdx\\
&\qquad=\eps_{n}\int_{0}^{\infty}\!\!\!\!\!\int_{\R}\px(\eta'(u_{\eps_{n},\,\beta_{n}})\px u_{\eps_{n},\,\beta_{n}})\phi dtdx -\eps_{n}\int_{0}^{\infty}\!\!\!\!\!\int_{\R} \eta''(u_{\eps_{n},\,\beta_{n}})(\px u_{\eps_{n},\,\beta_{n}})^2\phi dtdx\\
&\qquad\quad +\beta_{n}\int_{0}^{\infty}\!\!\!\!\!\int_{\R}\px(\eta'(u_{\eps_{n},\,\beta_{n}})\pxx u_{\eps_{n},\,\beta_{n}})\phi dtdx\\
&\qquad\quad- \beta_{n}\int_{0}^{\infty}\!\!\!\!\!\int_{\R}\eta''(u_{\eps_{n},\,\beta_{n}})\px u_{\eps_{n},\,\beta_{n}}\pxx u_{\eps_{n},\,\beta_{n}}\phi dtdx\\
&\qquad \le - \eps_{n}\int_{0}^{\infty}\!\!\!\!\!\int_{\R}\eta'(u_{\eps_{n},\,\beta_{n}})\px u_{\eps_{n},\,\beta_{n}}\px\phi dtdx - \beta_{n}\int_{0}^{\infty}\!\!\!\!\!\int_{\R}\eta'(u_{\eps_{n},\,\beta_{n}})\pxx u_{\eps_{n},\,\beta_{n}}\px\phi dtdx\\
&\qquad\quad - \beta_{n}\int_{0}^{\infty}\!\!\!\!\!\int_{\R}\eta''(u_{\eps_{n},\,\beta_{n}})\px u_{\eps_{n},\,\beta_{n}}\pxx u_{\eps_{n},\,\beta_{n}}\phi dtdx\\
&\qquad \le  \eps_{n}\int_{0}^{\infty}\!\!\!\!\!\int_{\R}\vert\eta'(u_{\eps_{n},\,\beta_{n}})\vert\vert\px u_{\eps_{n},\,\beta_{n}}\vert\vert\px\phi\vert dtdx\\
&\qquad\quad +\beta_{n}\int_{0}^{\infty}\!\!\!\!\!\int_{\R}\vert\eta'(u_{\eps_{n},\,\beta_{n}})\vert\vert\pxx u_{\eps_{n},\,\beta_{n}}\vert\vert\px\phi\vert dtdx\\
&\qquad\quad +\beta_{n}\int_{0}^{\infty}\!\!\!\!\!\int_{\R}\vert\eta''(u_{\eps_{n},\,\beta_{n}})\vert\vert\px u_{\eps_{n},\,\beta_{n}}\pxx u_{\eps_{n},\,\beta_{n}}\vert\vert\phi\vert dtdx\\
&\qquad\le  \eps_{n} \norm{\eta'}_{L^{\infty}(\R)}\norm{\px u_{\eps_{n},\,\beta_{n}}}_{L^2(\supp(\px\phi))}\norm{\px\phi}_{L^2(\supp(\px\phi))}\\
&\qquad\quad+ \beta_{n} \norm{\eta'}_{L^{\infty}(\R)}\norm{\pxx u_{\eps_{n},\,\beta_{n}}}_{L^2(\supp(\px\phi))}\norm{\px\phi}_{L^2(\supp(\px\phi))}\\
&\qquad\quad +\beta_{n} \norm{\eta''}_{L^{\infty}(\R)}\norm{\phi}_{L^{\infty}(\R)}\norm{\px u_{\eps_{n},\,\beta_{n}}\pxx u_{\eps_{n},\,\beta_{n}}}_{L^1(\supp(\px\phi))}\\
&\qquad\le  \eps_{n} \norm{\eta'}_{L^{\infty}(\R)}\norm{\px u_{\eps_{n},\,\beta_{n}}}_{L^2((0,T)\times\R)}\norm{\px\phi}_{L^2((0,T)\times\R)}\\
&\qquad\quad+ \beta_{n} \norm{\eta'}_{L^{\infty}(\R)}\norm{\pxx u_{\eps_{n},\,\beta_{n}}}_{L^2((0,T)\times\R)}\norm{\px\phi}_{L^2((0,T)\times\R)}\\
&\qquad\quad+\beta_{n} \norm{\eta''}_{L^{\infty}(\R)}\norm{\phi}_{L^{\infty}(\R^{+}\times\R)}\norm{\px u_{\eps_{n},\,\beta_{n}}\pxx u_{\eps_{n},\,\beta_{n}}}_{L^1((0,T)\times\R)}.
\end{align*}
\eqref{eq:u-entropy-solution} follows from \eqref{eq:beta-eps-1}, \eqref{eq:con3}, \eqref{eq:con4}, Lemmas \ref{lm:15} and \ref{lm:stima-l-6}.
\end{proof}
\begin{proof}[Proof of Theorem \ref{th:main}]
$i)$, $ii)$, and $iii)$ follow from Lemma \ref{lm:dist-solution}. $iv)$ and $v)$ follow from  Lemma \ref{lm:entropy-solution}, while \eqref{eq:umedianulla} follows from \eqref{eq:int-u-1} \eqref{eq:con1}, or \eqref{eq:con3}. Therefore, the proof is done.
\end{proof}

\end{document}